\documentclass[12pt]{amsart}

\usepackage{amsmath}
\usepackage{amssymb}
\usepackage{amsfonts}
\usepackage{amsthm}
\usepackage{enumerate}
\usepackage{hyperref}
\usepackage{color}

\theoremstyle{plain}
\newtheorem{thm}{Theorem}[section]

\newtheorem{lem}[thm]{Lemma}

\theoremstyle{definition}
\newtheorem{dfn}[thm]{Definition}

\newtheorem{rem}[thm]{Remark}

\newtheorem{dfns-rems}[thm]{Definitions and Remarks}
\newtheorem{notas-rems}[thm]{Notations and Remarks}
\newtheorem{exmps-rems}[thm]{Examples and Remarks}

\begin{document}


\title[Stanley depth of weakly polymatroidal ideals]{Stanley depth of weakly polymatroidal ideals}


\author[S. A. Seyed Fakhari]{S. A. Seyed Fakhari}

\address{S. A. Seyed Fakhari, School of Mathematics, Institute for Research
in Fundamental Sciences (IPM), P.O. Box 19395-5746, Tehran, Iran.}

\email{fakhari@ipm.ir}

\urladdr{http://math.ipm.ac.ir/fakhari/}


\begin{abstract}
Let $\mathbb{K}$ be a field and $S=\mathbb{K}[x_1,\dots,x_n]$ be the
polynomial ring in $n$ variables over the field $\mathbb{K}$. In this
paper, it is shown that Stanley's conjecture holds for $S/I$, if
$I$ is a weakly polymatroidal ideal.
\end{abstract}






\thanks{This research was in part supported by a grant from IPM (No. 93130422)}


\maketitle


\section{Introduction} \label{sec1}

Let $\mathbb{K}$ be a field and $S=\mathbb{K}[x_1,\dots,x_n]$ be the
polynomial ring in $n$ variables over the field $\mathbb{K}$. Let $M$ be a nonzero
finitely generated $\mathbb{Z}^n$-graded $S$-module. Let $u\in M$ be a
homogeneous element and $Z\subseteq \{x_1,\dots,x_n\}$. The $\mathbb
{K}$-subspace $u\mathbb{K}[Z]$ generated by all elements $uv$ with $v\in
\mathbb{K}[Z]$ is called a {\it Stanley space} of dimension $|Z|$, if it is
a free $\mathbb{K}[\mathbb{Z}]$-module. Here, as usual, $|Z|$ denotes the
number of elements of $Z$. A decomposition $\mathcal{D}$ of $M$ as a finite
direct sum of Stanley spaces is called a {\it Stanley decomposition} of
$M$. The minimum dimension of a Stanley space in $\mathcal{D}$ is called
{\it Stanley depth} of $\mathcal{D}$ and is denoted by ${\rm
sdepth}(\mathcal {D})$. The quantity $${\rm sdepth}(M):=\max\big\{{\rm
sdepth}(\mathcal{D})\mid \mathcal{D}\ {\rm is\ a\ Stanley\ decomposition\
of}\ M\big\}$$ is called {\it Stanley depth} of $M$. Stanley \cite{s}
conjectured that $${\rm depth}(M) \leq {\rm sdepth}(M)$$ for all
$\mathbb{Z}^n$-graded $S$-modules $M$. For a reader friendly introduction
to Stanley decomposition, we refer to \cite{psty} and for a nice survey on this topic we refer to \cite{h}.

Let $I$ be a monomial ideal of $S=\mathbb{K}[x_1,\dots,x_n]$ which is
generated in a single degree and assume that $G(I)$ is the set of minimal
monomial generators of $I$. The ideal $I$ is called {\it polymatroidal} if
the following exchange condition is satisfied: For monomials $u=x^{a_1}_1
\ldots x^{a_n}_ n$ and $v = x^{b_1}_1\ldots x^{b_n}_ n$ belonging to $G(I)$
and for every $i$ with $a_i > b_i$, one has $j$ with $a_j < b_j$ such that
$x_j(u/x_i)\in G(I)$.

In \cite{psy}, the authors prove that $S/I^k$ satisfies Stanley's conjecture for every polymatroidal ideal $I$ and every $k\gg 0$. In this paper, we prove a stronger result. We show that $S/I$ satisfies Stanley's conjecture for every polymatroidal ideal $I$. Since every power of a polymatroidal
ideal is again a polymatroidal ideal (see \cite[Theorem 12.6.3]{hh'}), this result implies that $S/I^k$ satisfy Stanley's conjecture, for every $k\geq 1$. In fact, we prove stronger result. We show that $S/I$ satisfies Stanley's conjecture, provided that $I$ is a weakly polymatroidal ideal (see Theorem \ref{main}).


\section{Main results} \label{sec2}

Weakly polymatroidal ideals are generalization of polymatroidal ideals and
they are defined as follows.

\begin{dfn} [\cite{mm}, Definition 1.1] \label{weak}
A monomial ideal $I$ of $S=\mathbb{K}[x_1,\dots,x_n]$ is called {\it weakly
polymatroidal} if for every two monomials $u = x_1^{a_1} \ldots x_n^{a_n}$
and $v = x_1^{b_1} \ldots x_n^{b_n}$ in $G(I)$ such that $a_1 = b_1,
\ldots, a_{t-1} = b_{t-1}$ and $a_t > b_t$ for some $t$, there exists $j >
t$ such that $x_t(v/x_j)\in I$.
\end{dfn}

Let $I$ be a monomial ideal and let $G(I)$ be the set of minimal monomial generators of $I$. Assume that $u_1\prec u_2 \prec \ldots \prec u_t$ is a linear order on $G(I)$. We say that $I$ has {\it linear quotients with respect to $\prec$}, if for every $2\leq i\leq t$, the ideal $(u_1, \ldots, u_{i-1}):u_i$ is generated by a subset of variables. We say that $I$ has {\it linear quotients}, if it has linear quotients with respect to a linear order on $G(I)$. It is known that every weakly polymatroidal ideal $I$ has linear quotients with respect to the pure lexicographical
order $<_{\rm purelex}$ induced by $x_1 > x_2 > \ldots > x_n$ (see \cite[Theorem 12.7.2]{hh'}).
The following result due to Sharifan and Varbaro has crucial role in the proof of our main result.

\begin{thm} [\cite{sv}, Corollary 2.7] \label{pd}
Let $I\subseteq S=\mathbb{K}[x_1,\dots,x_n]$ be a monomial ideal. Assume that $I$ has linear quotients with respect to $u_1\prec u_2 \prec \ldots \prec u_t$, where $\{u_1, \ldots, u_t\}$ is the set of minimal
monomial generators of $I$. For every $2\leq i\leq t$, let $n_i$ be the number of variables which generate $(u_1, \ldots, u_{i-1}):u_i$. Then $${\rm pd}_S(S/I)={\rm max}\{n_i: 2\leq i\leq t\}+1.$$
\end{thm}

Keeping the notations of Theorem \ref{pd} in mind, Auslander--Buchsbaum formula implies that $${\rm depth}_S(S/I)=n-{\rm max}\{n_i: 2\leq i\leq t\}-1.$$

Let $I$ be a weakly polymatroidal ideal. In order to prove Stanley's conjecture for $S/I$, we need the following lemma. It shows that the depth of a weakly polymatroidal ideal does not decrease under the elimination of $x_1$. As usual for every monomial $u$, the {\it support} of $u$, denoted by ${\rm Supp}(u)$, is the set of variables which divide $u$.

\begin{lem} \label{del}
Let $I$ be a weakly polymatroidal ideal ideal of $S=\mathbb{K}[x_1,\ldots,x_n]$, such that $$x_1\in \bigcup_{u\in G(I)}{\rm Supp}(u).$$Let $S'=\mathbb{K}[x_2, \ldots, x_n]$ be the polynomial ring obtained from $S$ by deleting the variable $x_1$ and consider the ideal $I'=I\cap S'$. Then ${\rm depth}_{S'}(S'/I')\geq {\rm depth}_S(S/I)$.
\end{lem}
\begin{proof}
We first note that $I\neq 0$, sine $$x_1\in \bigcup_{u\in G(I)}{\rm Supp}(u).$$If $I'=0$, then ${\rm depth}_{S'}(S'/I')=n-1\geq {\rm depth}_S(S/I)$. Therefore, assume that $I'\neq 0$. Let $G(I)=\{u_1, \ldots, u_t\}$ be the set of minimal monomial generators of $I$ and suppose that $u_t <_{\rm purelex} u_{t-1}<_{\rm purelex} \ldots <_{\rm purelex} u_1$. Since $I'\neq 0$, the assumption implies that there exists an integer $1\leq s \leq t-1$ such that $x_1$ divides $u_1, \ldots, u_s$ and $x_1$ does not divide $u_{s+1}, \ldots, u_t$. It is clear that $I'$ is a weakly polymatroidal ideal which is minimally generated by the set $\{u_{s+1}, \ldots, u_t\}$. It follows from the definition of weakly polymatroidal ideals that for every $s+1\leq i \leq t$, there exists an integer $i_j\geq 2$ such that $x_1(u_i/x_{i_j})\in I$. Thus, there exists a monomial, say $u_{i_j}\in G(I)$, such that $u_{i_j}$ divides $x_1(u_i/x_{i_j})$. It is obvious that $1\leq i_j\leq s$, because otherwise $u_{i_j}$ divides $u_i/x_{i_j}$ which is a contradiction. Hence, for every $i$ with $s+2\leq i \leq t$ we have$$(x_1)+((u_{s+1}, \ldots, u_{i-1}):u_i)\subseteq (u_{i_j}, u_{s+1}, \ldots, u_{i-1}):u_i\subseteq (u_1, \ldots, u_{i-1}):u_i.$$On the other hand, it is clear that$$x_1\notin (u_{s+1}, \ldots, u_{i-1}):u_i$$Therefore, by Theorem \ref{pd}, we conclude that ${\rm pd}_S(S/I)\geq {\rm pd}_{S'}(S'/I')+1$. Now Auslander--Buchsbaum completes the proof of the Lemma.
\end{proof}

We are now ready to prove the main result of this paper.

\begin{thm} \label{main}
Let $I$ be a weakly polymatroidal ideal of $S=\mathbb{K}[x_1,\ldots,x_n]$ such that $I\neq S$. Then $S/I$ satisfies Stanley's conjecture.
\end{thm}
\begin{proof}
We prove the assertion by induction on $n$ and $$\sum_{u\in G(I)} {\rm deg} (u),$$ where $G(I)$ is the set of minimal monomial generators of $I$. If
$n=1$ or $$\sum_{u\in G(I)} {\rm deg} (u)=1,$$ then $I$ is a principal ideal and so we have ${\rm depth}(S/I)=n-1$ and by \cite[Theorem 1.1]{r},
${\rm sdepth}(S/I)=n-1$. Therefore, in
these cases, the assertion is trivial.

We now assume that $n\geq 2$ and
$$\sum_{u\in G(I)} {\rm deg} (u)\geq 2.$$Let $S'=\mathbb{K}[x_2, \ldots, x_n]$ be the polynomial ring obtained from $S$ by deleting the variable $x_1$ and consider the ideals $I'=I\cap S'$ and
$I''=(I:x_1)$. If $$x_1\notin \bigcup_{u\in G(I)}{\rm Supp}(u),$$then the induction hypothesis on $n$ implies that
${\rm depth}(S/I)={\rm depth}(S'/I')+1$. On the other hand, by \cite[Lemma 3.6]{hvz}, we conclude that ${\rm sdepth}(S/I)={\rm
sdepth}(S'/I')+1$.
Therefore, using the induction hypothesis on $n$ we conclude that ${\rm sdepth}(S/I)\geq {\rm depth}(S/I)$.
Therefore, we may assume that $$x_1\in \bigcup_{u\in G(I)}{\rm Supp}(u).$$If $x_1\in I$ (which means $x_1\in G(I)$), then ${\rm depth}(S/I)={\rm depth}(S'/I')$. On the other hand it follows from \cite[Theorem
1.1]{r} and \cite[Lemma 3.6]{hvz} that ${\rm sdepth}(S/I)={\rm
sdepth}(S'/I')$. Therefore the induction hypothesis on $n$ implies that ${\rm sdepth}(S/I)\geq {\rm depth}(S/I)$. Thus we assume that $x_1\notin I$.

Now $S/I=(S'/I'S')\oplus x_1(S/I''S)$ and therefore by definition of  the Stanley depth we have

\[
\begin{array}{rl}
{\rm sdepth}(S/I)\geq \min \{{\rm sdepth}_{S'}(S'/I'S'), {\rm sdepth}_S(S/I'')\}.
\end{array} \tag{1} \label{1}
\]

Using Lemma \cite[Lemma 2.2]{s2}, it follows that $I''$ is a weakly polymatroidal ideal. Hence \cite[Corollary 1.3]{r2} together with the induction hypothesis on $$\sum_{u\in G(I)} {\rm deg}(u)$$ implies that $${\rm sdepth}_S(S/I'')\geq {\rm depth}_S(S/(I:x_1))\geq {\rm depth}_S(S/I).$$

On the other hand, $I'S'$ is a weakly polymatroidal ideal and since $$x_1 \in \bigcup_{u\in G(I)} {\rm Supp}(u),$$using Lemma \ref{del}, we conclude that ${\rm depth}_{S'}(S'/I')\geq {\rm depth}_S(S/I)$ and therefore by
the induction hypothesis on $n$, we conclude that ${\rm sdepth}_{S'}(S'/I'S')\geq {\rm depth}_S(S/I)$. Now the assertions follow by inequality (\ref{1}).
\end{proof}

\begin{rem}
Let $I$ be a weakly polymatroidal ideal of $S$. In Theorem \ref{main}, we showed that $S/I$ satisfies Stanley's conjecture. It is natural to ask whether $I$ itself satisfies Stanley's conjecture. The answer of this question is positive. Indeed, Soleyman Jahan \cite{so} proves that Stanley's conjecture holds true for every monomial ideal with linear quotients.
\end{rem}





\end{document}